\documentclass[reqno]{amsart}
\usepackage{cases}
\usepackage{latexsym}
\usepackage{amsmath}
\usepackage[arrow,matrix]{xy}
\usepackage{stmaryrd}
\usepackage{amsfonts}
\usepackage{amsmath,amssymb,amscd,bbm,amsthm,mathrsfs,dsfont}

\usepackage{fancyhdr}
\usepackage{amsxtra,ifthen}
\usepackage{verbatim}

\numberwithin{equation}{section}

\theoremstyle{plain}
\newtheorem{theorem}{Theorem}[section]
\newtheorem{lemma}[theorem]{Lemma}
\newtheorem{proposition}[theorem]{Proposition}
\newtheorem{corollary}[theorem]{Corollary}

\theoremstyle{definition}
\newtheorem{definition}[theorem]{Definition}

\newtheorem{remark}[theorem]{Remark}

\newcommand{\I}{\mathcal I}
\newcommand{\one}{\mathbf 1}
\newcommand{\ad}{\operatorname{ad}}
\newcommand{\Span}{\operatorname{span}}

\begin{document}

\title[FIs on incidence algebras]
{Functional identities of degree 2 at two-sided zero products on incidence algebras}

\author{Hongyu Jia and Zhankui Xiao}

\address{Jia: School of Mathematical Sciences, Huaqiao University,
Quanzhou, Fujian, 362021, P. R. China}
\email{jiahy1995@163.com}

\address{Xiao: School of Mathematical Sciences, Huaqiao University,
Quanzhou, Fujian, 362021, P. R. China}
\email{zhkxiao@hqu.edu.cn}

\begin{abstract}
Let $R$ be a commutative ring with unity such that $\frac{1}{2}\in R$.
Let $X$ be a connected finite poset with $|X|>2$ and $I(X,R)$ be the incidence
algebra of $X$ over $R$. In this paper, we characterize the forms of linear maps
$F_1,F_2,F_3,F_4:I(X,R)\to I(X,R)$ satisfying
\[
F_1(f)g+fF_2(g)+F_3(g)f+gF_4(f)=0,
\]
whenever $fg=gf=0$. We prove that the $F_i$'s are of the so-called standard form
if and only if any two edges in the comparability graph of $X$ are contained in one cycle.
The ingredients of the proof contain a characterization of $2$-connectedness in comparability graph
and the two-sided zero product determined property of incidence
algebras.
\end{abstract}



\thanks{The first author is partially supported by the National Natural Science Foundation of China (No. 12401022).}

\maketitle

\section{Introduction}\label{sec:introduction}

Let $m,n$ be two non-negative integers and
$f(x_1, \ldots, x_m, y_1, \ldots, y_n)$ be
a non-commutative polynomial in the variables $x_1, \ldots, x_m, y_1, \ldots, y_n$
with integral coefficients. Given a ring $R$ and a nonempty subset $S$ of $R$,
we say that $f$ is a {\em functional identity} (FI for short) on $S$ if
$$
f(s_1, \ldots, s_m, F_1(s_1, \ldots, s_m), \ldots, F_n(s_1, \ldots, s_m))=0
$$
for all $s_1, \ldots, s_m\in S$, where $F_i: S^m\rightarrow R$ are maps (called {\em functions}) for $1\leq i\leq n$.
In this case the functions $F_1, \ldots, F_n$ are considered as unknowns, and
the main objective of FI theory is to describe the form of the functions $F_1, \ldots, F_n$.
From this point of view, the FI theory can be roughly said to be a theory of solving equations
on (usually non-commutative) rings.

Over thirty years ago, Bre\v{s}ar studied some basic FIs related to additive commuting maps
and commuting traces of bi-additive maps and their applications in Lie theory \cite{Bre93-1,Bre93-2}.
In fact, the main motivation of developing the FI theory was searching for tools for the
Herstein's Lie-type mapping conjectures (see \cite{Her}). In the process of establishing the FI theory,
the FIs of degree 2 are of special importance. Let $R$ be an associative ring
and $F_1, F_2, F_3, F_4: R\rightarrow R$ be maps such that
\begin{equation}\label{eq:introduction-FI}
F_1(x)y+xF_2(y)+F_3(y)x+yF_4(x)=0
\end{equation}
for all $x,y\in R$. This is the FI of degree 2 studied in \cite{Bre95-2}.
It is clear that an additive commuting map satisfies the identity \eqref{eq:introduction-FI}.
We would like to remark that the general FIs of degree 2 on
prime rings are studied in \cite{Bre95-1} and this is the paper where the phrase ``functional identity" was introduced.
In 1998 Chebotar \cite{Che98} extended the results of \cite{Bre95-1} to FIs of arbitrary degree. In particular, he initiated
the direction in which the FI theory was later created. Motivated by \cite{Che98}, Beidar provided a systematic
approach in his fundamental paper \cite{Beidar98}, which in our opinion can be viewed as the first framework for the general FI theory.
A remarkable application of the FI theory is the complete solutions of Herstein's conjectures on
Lie homomorphisms and Lie derivations in associative rings, see \cite{BBCM1,BBCM2,BBCM3}.
We refer the reader to the monograph \cite{BCM07} and the recent survey \cite{Bre23} for a comprehensive
understanding of the FI theory.

Notice that the current FI theory depends on a core concept of $d$-free sets, but
there are also many results on FIs that are not superseded by the general theory (see \cite{BCM07}).
Triangular algebras are typical examples that are not $d$-free, on which commuting maps
and their applications to Lie automorphisms can nevertheless be described \cite{BE04,Cheung}.
Subsequently, the FI (\ref{eq:introduction-FI}) was studied by Eremita \cite{Eremita13}
and he described the form of additive maps $F_i$'s on triangular rings satisfying certain conditions.
Then the results in \cite{Eremita13} was generalized by Wang in \cite{Wang15}.
More recently, Arga\c{c} and Ghahramani \cite{ArgacGhahramani25} studied a local version of the FI
\eqref{eq:introduction-FI} under the condition $xy=yx=0$. They
obtained its standard solutions on certain triangular algebras,
with applications to upper triangular matrix algebras, finite nest
algebras, and block upper triangular matrix algebras.
The above results (and more results can be found in \cite{BCM07,Bre23}) suggest that
there may be an FI theory in a broader sense, and this is the initial motivation of our present work.

The specific objective of this paper is to study the FI \eqref{eq:introduction-FI} with condition $xy=yx=0$
on incidence algebras, another natural generalization of upper triangular matrix algebras.
About ten years ago, the second author \cite{Xiao}, Zhang and Khrypchenko \cite{ZhangKh}
started to study the Herstein's Lie-type mapping research program on incidence algebras in a linear and combinatorial manner.
In \cite{jia20} the authors proved that every commuting map on the incidence algebra $I(X,R)$ is proper
if and only if any two edges in the comparability graph $\Gamma(X)$ (see Section 2) are contained in one cycle.
On the other hand, the structure of Lie automorphisms was studied by Fornaroli, Khrypchenko and Santulo \cite{FKhS22}.
For more results about Lie-type maps of incidence algebras, we refer the reader
to \cite{FKhS24,KaKh21,WangXiao} and the references therein.
Therefore, it is natural to study the general FIs and their applications in Lie theory on incidence algebras.

An incidental purpose of this paper is to determine the two-sided zero
product property of incidence algebras, since the local condition $xy=yx=0$ is naturally connected with the
notion of a two-sided zero product determined algebra.
This notion was introduced and systematically investigated over fields by
Bajuk and Bre\v{s}ar in \cite{BajukBresar22}. When $X$ is a finite partially ordered set (poset for short),
we shall show that the incidence algebra $I(X,R)$ is two-sided zero product determined and
hence is zero Lie product determined.

The paper is organized as follows. Section~2, we introduce the
comparability graph of a finite poset, incidence algebras,
and some required facts of derivations. Especially, in a comparability graph several equivalent descriptions
of the fact that any two edges are contained in one cycle are shown for later use.
Section~3 is devoted to proving the two-sided zero product determined property of incidence algebras.
Section~4, we study the FI \eqref{eq:introduction-FI} on incidence algebras by characterizing
the symmetric and anti-symmetric versions of it separately. Then a sufficient and necessary condition
on comparability graphs is given for the fact the FI \eqref{eq:introduction-FI} has the so-called standard solutions.

\section{Preliminaries}\label{sec:preliminaries}

In this section, we introduce some results (some of which are new) about
comparability graphs, incidence algebras and derivations.
{\em From now on, we always denote $R$ a commutative ring with unity}.

\subsection{Comparability graph}

We first recall some basic graph-theoretic terminologies. Let $G$ be
a finite simple undirected graph with vertex set $V(G)$ and edge
set $E(G)$. A \emph{path} from a vertex $u$ to a vertex $v$ in $G$
is a sequence of distinct vertices
\[
u=x_0,x_1,\ldots,x_m=v
\]
such that $\{x_{i-1},x_i\}\in E(G)$ for every $1\leq i\leq m$.
The graph $G$ is called \emph{connected} if every two distinct
vertices of $G$ can be joined by a path. For a subset $S\subseteq V(G)$, the notation $G\setminus S$
denotes the graph obtained from $G$ by deleting all vertices in
$S$ and all edges incident with them.

Let $(X,\leq)$ be a finite poset. For $i\leq j\in X$ with $i\neq j$, we write $i<j$ or $j>i$ for short.
For distinct $x,y\in X$, denote $x\sim y$ if $x<y$ or $y<x$.
The \emph{comparability graph} $\Gamma(X)$ associated to $X$ (see \cite[Chapter 5]{Gol})
is the simple undirected graph with vertex set $X$ and
edge set
\[
E(\Gamma(X))
 =\bigl\{\{x,y\} \mid x,y\in X,\ x\sim y\bigr\}.
\]
The poset $X$ is called \emph{connected} if its comparability graph
$\Gamma(X)$ is connected.

We would like to remind the reader to distinguish the well-known Hasse diagram (see \cite[\S 1.1]{SpDo})
and the comparability graph associated to a poset $X$. For example, let $X=\{1,2,3,4\}$
with partial order $1<2$, $2<3$ and $2<4$. Then the associated Hasse diagram is the Dynkin diagram of
type $D_4$, but the edge set of comparability graph $\Gamma(X)$ is
$\{\{1,2\}, \{2,3\}, \{1,3\}, \{2,4\}, \{1,4\}\}$.

\begin{definition}\label{def:cut-vertex}
Let $X$ be a connected finite poset. A vertex $z\in X$ is called a
\emph{cut vertex} if $\Gamma(X)\setminus\{z\}$ is disconnected.
\end{definition}

For example, let $X=\{x,y,z\}$ with $x<y<z$. By transitivity
$x<z$. Therefore, after deleting $y$, the vertices $x$ and $z$
remain joined by the edge $\{x,z\}$. Hence
$\Gamma(X)\setminus\{y\}$ is still connected, and consequently $y$ is
not a cut vertex.

\begin{definition}\label{def:cycle}
Let $x_1,\ldots,x_n$ be distinct vertices of $X$, where $n\geq3$.
We say that $\{x_1,\ldots,x_n\}$ forms a \emph{cycle} in
$\Gamma(X)$ if
\[
x_1\sim x_2,\quad x_2\sim x_3,\quad\ldots,\quad
x_{n-1}\sim x_n,\quad x_n\sim x_1.
\]
\end{definition}

If $\{x_1,x_2,\ldots,x_n\}$ forms a cycle, we say that the cycle contains edges
$\{x_{i-1},x_i\}$, $1\leq i\leq n$, where the subscripts are modulo $n$.
For any two edges $\{x,y\}$ and $\{u,v\}$ in $E(\Gamma(X))$, define
$
\{x,y\}\approx\{u,v\}
$
if and only if there exists a cycle in $\Gamma(X)$ containing both
$\{x,y\}$ and $\{u,v\}$. We also set $\{x,y\}\approx\{x,y\}$. Then
the binary relation $\approx$ is in fact an equivalence relation on
$E(\Gamma(X))$ (see \cite[Lemma 2.2]{Yang19}).

\begin{definition}\label{def:k-connected}
Let $G$ be a finite graph and let $k\geq1$. The graph $G$ is called
\emph{$k$-connected} if $|V(G)|>k$ and $G\setminus S$ is connected
for every subset $S\subseteq V(G)$ satisfying $|S|<k$.

A finite poset $X$ is called \emph{$k$-connected} if the associated
comparability graph $\Gamma(X)$ is $k$-connected.
\end{definition}

The following lemma describes the intrinsic relationship among the aforementioned concepts,
which is helpful for our investigation of FIs on incidence algebras.

\begin{lemma}\label{lem:graph-equivalences}
Let $X$ be a finite connected poset with $|X|>2$. The following statements are
equivalent:
\begin{enumerate}
\item[(a)] Any two edges in $E(\Gamma(X))$ are contained in one cycle.
\item[(b)] $X$ is $2$-connected.
\item[(c)] $\Gamma(X)\setminus\{z\}$ is connected for every $z\in X$.
\item[(d)] $\Gamma(X)$ has no cut vertex.
\end{enumerate}
\end{lemma}

\begin{proof}
Clearly it follows from the definitions that the statements (b), (c) and (d) are equivalent.
We now prove the equivalence of (a) and (b).

Assume first that (a) holds. We aim to prove that
$\Gamma(X)\setminus\{z\}$ is connected for every $z\in X$.
Suppose that, to the contrary, there exists a vertex $z\in X$ such that
$\Gamma(X)\setminus\{z\}$ is disconnected. Let $Y_1$ and $Y_2$ be two
distinct connected components of $\Gamma(X)\setminus\{z\}$. Since $\Gamma(X)$ is
connected, each $Y_i$ contains a vertex $u_i$ adjacent to $z$.
Indeed, a path in $\Gamma(X)$ from a vertex of $Y_i$ to $z$ must enter $z$
from a vertex of $Y_i$. Thus $\{z,u_i\}\in E(\Gamma(X))$ for $i=1,2$.
By condition (a), the two edges $\{z,u_1\}$ and
$\{z,u_2\}$ are contained in a cycle $C$. Since a vertex on a
cycle is incident with exactly two edges of that cycle,
deleting $z$ from $C$ leaves a path from $u_1$ to $u_2$ in
$\Gamma(X)\setminus\{z\}$. This contradicts the fact that $u_1$ and $u_2$
belong to distinct connected components.
Therefore (c) and hence (b) holds.

Assume now that (b) holds. We first claim that every edge of $\Gamma(X)$ is contained
in a cycle. In fact, for an arbitrary edge $\varepsilon=\{a,b\}\in E(\Gamma(X))$, the vertex
$a$ must have a neighbor $c\neq b$. Otherwise, $b$ would be the
only neighbor of $a$, so deleting $b$ would isolate $a$. Since
$|X|\geq 3$, at least one further vertex would remain, and
$\Gamma(X)\setminus\{b\}$ would be disconnected. This contradicts the
$2$-connectedness of $\Gamma(X)$.
Since $\Gamma(X)\setminus\{a\}$ is connected, it contains a path
$P$ from $b$ to $c$. Notice that the path $P$ does not contain $a$. Therefore
the union of $P$ with the two edges $\{a,b\}$ and $\{a,c\}$ forms a
cycle containing $\varepsilon$.

Now let $\varepsilon_1$ and $\varepsilon_2$ be two distinct edges
of $\Gamma(X)$, where $\varepsilon_i=\{a_i,b_i\}$. Next we prove that
$\varepsilon_1$ and $\varepsilon_2$ are contained in one cycle.
Introduce a new vertex $w_i$ into each edge $\varepsilon_i$ and replace $\varepsilon_i$
with the two edges $\{a_i,w_i\}$ and $\{w_i,b_i\}$.
This operation is called the subdivision of $\varepsilon_i$.
Denote the resulting graph by $\widetilde \Gamma(X)$.

We claim that the new graph $\widetilde \Gamma(X)$ is $2$-connected. Since $\Gamma(X)$ is
$2$-connected, by (c), $\Gamma(X)\setminus\{z\}$ is connected for every $z\in X$.
This means that $\widetilde \Gamma(X)\setminus\{z\}$ is connected for every $z\in X\setminus\{a_1,w_1,b_1,a_2,w_2,b_2\}$.
If $z=w_i$, then $\widetilde \Gamma(X)\setminus\{w_i\}$ is connected following from the connectedness of $\Gamma(X)$.
If $z\in \{a_1,b_1,a_2,b_2\}$, without loss of generality, we suppose that $z=a_i$ with $i=1,2$.
After deleting $z$, the edge $\{a_i, w_i\}$ is removed, whereas the edge $\{w_i, b_i\}$ remains.
Since $b_i\neq z$, the vertex $b_i$ belongs to the connected graph $\Gamma(X)\setminus\{z\}$.
Consequently, $w_i$ remains connected to every remaining original vertex through the edge $\{w_i,b_i\}$.
Thus $\widetilde \Gamma(X)\setminus\{z\}$ is connected for every $z\in X\cup \{w_1,w_2\}$.
Then $\widetilde \Gamma(X)$ is $2$-connected.

By the global vertex form of Menger's theorem (see
\cite[Theorem~3.3.6 (i)]{Diestel}), any two distinct vertices of a
$2$-connected graph can be joined by two independent paths.
Here, the phrase `two independent paths' means that the two paths have no common vertices other than their endpoints.
Applying this result to $w_1$ and $w_2$, we obtain two independent paths from $w_1$ to $w_2$.
 Their union is a cycle containing both $w_1$ and $w_2$.
Since $w_i$ in $\widetilde \Gamma(X)$ only connected to $a_i,b_i$ by two edges $\{a_i,w_i\}$ and $\{w_i,b_i\}$. Hence the cycle through
$w_i$ must contain both edges $\{a_i,w_i\}$ and $\{w_i,b_i\}$.
Deleting $w_i$ from this cycle and replacing the path
$a_i\sim w_i\sim b_i$ by the original edge
$\varepsilon_i=\{a_i,b_i\}$ produces a  cycle in $\Gamma(X)$
containing both $\varepsilon_1$ and $\varepsilon_2$.
Therefore (a) holds.
\end{proof}

\subsection{Incidence algebras}

Let $X$ be a finite poset. The {\em incidence algebra} $I(X,R)$ of $X$ over $R$ is defined to be the set of functions (see \cite{SpDo})
\[
I(X,R):=\{f: X\times X\longrightarrow R \mid f(x,y)=0\ \text{if}\ x\nleq y\},
\]
endowed with the usual $R$-module structure and the multiplication given by {\em convolution}
\[
(fg)(x,y)=\sum_{x\leq z\leq y}f(x,z)g(z,y)
\]
for all $f,g\in I(X,R)$ and $x,y\in X$. For $x\leq y$, let $e_{xy}$ be the function taking value $1$ at
$(x,y)$ and $0$ at all other pairs. We also set $e_x=e_{xx}$. Then
\begin{equation}\label{eq:matrix-unit-product}
e_{xy}e_{uv}=\delta_{yu}e_{xv},~\text{for~all}~x\leq y,\ u\leq v,
\end{equation}
by the definition of convolution, where $\delta_{yu}$ denotes the Kronecker delta.
Notice that the poset $X$ is assumed to be finite. Then the incidence algebra $I(X,R)$
is isomorphic to a subalgebra of the upper triangular matrix algebra ${\rm T}_{|X|}(R)$
and the set $\{e_{xy}\mid x\leq y\}$ forms an $R$-basis of $I(X,R)$, see \cite[Proposition 1.2.7]{SpDo}.

The following observation is important.

\begin{lemma}\label{lem:idempotent-span}
The incidence algebra $I(X,R)$ is $R$-linear spanned by idempotents.
\end{lemma}

\begin{proof}
For $x<y$, there is
\[
(e_x+e_{xy})^2
=e_x^2+e_xe_{xy}+e_{xy}e_x+e_{xy}^2
=e_x+e_{xy}.
\]
Thus $e_x+e_{xy}$ is an idempotent and
$e_{xy}=(e_x+e_{xy})-e_x$ is a linear combination of two
idempotents. Hence the whole algebra is $R$-linear spanned by idempotents.
\end{proof}

Let us write the unity of $I(X,R)$ as
$\one=\sum_{x\in X}e_x$. We next recall the standard description of the center of an
incidence algebra, see \cite[Corollary 1.3.14]{SpDo}.

\begin{proposition}\label{prop:center}
Let $X_1,\ldots,X_m$ be the connected components of
$\Gamma(X)$ and set $\one_i=\sum_{x\in X_i}e_x$.
Then the center
\[
Z(I(X,R))=\bigoplus_{i=1}^m R\one_i.
\]
In particular, if $X$ is connected, then $Z(I(X,R))=R\one$.
\end{proposition}

Now we write $\I=I(X,R)$ for convenience. For any $f,g\in\I$, as usual we set $[f,g]=fg-gf$ for their commutator and
$f\circ g=fg+gf$ for their Jordan product. The $R$-submodule
generated by all commutators is denoted by
$[\I,\I]=R$-$\Span\{[f,g]\mid f,g\in\I\}$.
The following two facts will be used repeatedly in the
subsequent computations. Although their proofs can be found
elsewhere in the literature, we include brief proofs here for completeness.

\begin{lemma}\label{lem:commutator-submodule}
We have
$[\I,\I]=R$-$\Span\{e_{xy}\mid x<y\}$.
\end{lemma}

\begin{proof}
For $x<y$, it follows from \eqref{eq:matrix-unit-product} that
$[e_x,e_{xy}]=e_xe_{xy}-e_{xy}e_x=e_{xy}\in [\I,\I]$.
Conversely, for any $f,g\in\I$, the convolution shows that
$[f,g](x,x)=f(x,x)g(x,x)-g(x,x)f(x,x)=0$ for all $x\in X$, since $X$ is a poset.
Hence every commutator is a linear combination of the
$e_{xy}$'s with $x<y$.
\end{proof}

\begin{lemma}\label{lem:coefficient-formulas}
Let $f=\sum_{u\leq v}f(u,v)e_{uv}\in\I$ and $x\in X$.  Then
\begin{align}
[f,e_x]&=\sum_{u<x}f(u,x)e_{ux}
-\sum_{x<v}f(x,v)e_{xv},
\label{eq:comm-ex}\\
e_x\circ f&=2f(x,x)e_x
+\sum_{u<x}f(u,x)e_{ux}
+\sum_{x<v}f(x,v)e_{xv}.
\label{eq:jordan-ex}
\end{align}
\end{lemma}

\begin{proof}
Using \eqref{eq:matrix-unit-product} term by term, we obtain
\[
e_xf
=\sum_{u\leq v}f(u,v)e_xe_{uv}
=\sum_{x\leq v}f(x,v)e_{xv}.
\]
Similarly,
\[
fe_x
=\sum_{u\leq v}f(u,v)e_{uv}e_x
=\sum_{u\leq x}f(u,x)e_{ux}.
\]
Subtracting these two identities gives \eqref{eq:comm-ex} and adding them gives \eqref{eq:jordan-ex}.
\end{proof}

\subsection{Derivations}

An $R$-linear map $D:\I\to\I$ is called a \emph{derivation} if
\[
D(fg)=D(f)g+fD(g)
\]
for all $f,g\in\I$. For each $h\in\I$, define
$\ad_h:\I\to\I$ by $\ad_h(f)=[h,f]$. Then $\ad_h$ is a
derivation, called the \emph{inner derivation} induced by $h$.

\begin{definition}\label{def:transitive}
An element $\sigma\in \I$ is called \emph{transitive} if
$\sigma(x,z)=\sigma(x,y)+\sigma(y,z)$
whenever $x\leq y\leq z$.
\end{definition}

\begin{lemma}{\rm(\cite[Lemma 2.3]{Xiao})}\label{lem:transitive-derivation}
Let $\sigma$ be transitive. The $R$-linear map $\Delta_\sigma:\I\to\I$
defined by $\Delta_\sigma(e_{xy})=\sigma(x,y)e_{xy}$, for all $x\leq y$, is a
derivation.
\end{lemma}

The derivation $\Delta_\sigma$ constructed in
Lemma~\ref{lem:transitive-derivation} is called the
\emph{transitive derivation induced by $\sigma$}.
Notice that in some literatures a transitive derivation is also known as an additive derivation,
see \cite[Section 7.1]{SpDo} for example. When nonlinear derivations are studied on incidence algebras,
an additive derivation will appear naturally, where it means an additive map $D$ of $\I$ satisfying
the Leibniz formula $D(fg)=D(f)g+fD(g)$, see \cite{Yang19} for example.
Therefore, we adopt the concept of transitive derivations here to avoid confusion.

We end this section with a technical lemma.

\begin{lemma}\label{lem:derivation-diagonal}
If $D:\I\to\I$ is a derivation, then
$D(e_x)(y,y)=0$ for all $x,y\in X$.
\end{lemma}

\begin{proof}
In fact, the statement follows from \cite[Theorem 2.2]{Xiao}.
We also give a proof here for reader's convenience. Notice that
$D(e_x)=D(e_x)e_x+e_xD(e_x)$ for all $x\in X$.
Multiplying this identity by $e_y$ from both sides,
we can get $D(e_x)(y,y)=0$ for all $x,y\in X$.
\end{proof}

\section{Two-sided zero product determined incidence algebras}\label{sec:two-sided-zpd}

Let $F_1, F_2, F_3, F_4: \I \rightarrow \I$ be $R$-linear maps such that
\[
F_1(f)g+fF_2(g)+F_3(g)f+gF_4(f)=0
\]
whenever $fg=gf=0$. In this paper, our specific objective is to
describe the form of $F_i$'s on incidence algebras. Let us define the $R$-bilinear map
$\varphi: \I\times\I \rightarrow \I$ by $\varphi(f,g):=F_1(f)g+fF_2(g)+F_3(g)f+gF_4(f)$.
Then $\varphi(f,g)=0$, whenever $fg=gf=0$. Hence our objective is naturally connected with the
notion of a two-sided zero product determined algebra \cite{BajukBresar22}.
The purpose of this section is to prove that the incidence algebra $\I$ is two-sided zero product determined.

\begin{definition}\label{def:two-sided-zpd}
An $R$-algebra $\mathcal A$ is called \emph{two-sided zero product
determined} ($2$-zpd for short), if for every $R$-module
$V$ and every bilinear map $\Phi:\mathcal A\times\mathcal A\to V$
satisfying
\[
\Phi(a,b)=0,~\text{whenever}~ab=ba=0,
\]
then there exist $R$-linear maps $\tau_1,\tau_2:\mathcal A\to V$ such that
\[
\Phi(x,y)=\tau_1(xy)+\tau_2(yx),~\text{for~all}~x,y\in\mathcal A.
\]
\end{definition}

We start our proof by recalling a result on bilinear maps over algebras spanned linearly by idempotents.

\begin{lemma}{\rm (\cite[Theorem 3.5]{Ghahramani14})} \label{lem:idempotent-bilinear}
Let $\mathcal A$ be an $R$-algebra with unity $\one$ which is spanned as an $R$-module by
its idempotents. Let $V$ be an $R$-module and
$\Phi:\mathcal A\times\mathcal A\to V$ be a bilinear map such that
\[
\Phi(a,b)=0,~\text{whenever}~ab=ba=0.
\]
Then, for all $x,y\in\mathcal A$,
\begin{align}
\Phi(x,\one)&=\Phi(\one,x),
\label{eq:phi-unit-symmetry}\\
\Phi(x,y)+\Phi(y,x)&=\Phi(xy+yx,\one).
\label{eq:phi-jordan-unit}
\end{align}
\end{lemma}

It is clear that the incidence algebra $\I$ satisfies all the hypotheses of
Lemma~\ref{lem:idempotent-bilinear} by Lemma~\ref{lem:idempotent-span}.
Let $V$ be an $R$-module and let $\Phi:\I\times\I\to V$ be an arbitrary bilinear map with
$\Phi(f,g)=0$ whenever $fg=gf=0$.
Define
\begin{equation}\label{eq:psi-definition}
\Psi(f,g):=\Phi(f,g)-\Phi(fg,\one).
\end{equation}

\begin{lemma}\label{lem:psi-properties}
The bilinear map $\Psi$ defined in \eqref{eq:psi-definition} satisfies
\begin{align}
\Psi(f,g)&=0,~\text{whenever}~fg=gf=0;\notag\\
\Psi(f,g)+\Psi(g,f)&=0,~\text{for~all}~f,g\in\I;\notag\\
\Psi(e_x,e_x)&=0,~\text{for~all}~x\in X.\notag
\end{align}
\end{lemma}

\begin{proof}
By \eqref{eq:psi-definition}, the first statement is clear.
Combining Lemma~\ref{lem:idempotent-span} and Lemma~\ref{lem:idempotent-bilinear}, we have
$\Phi(f,g)+\Phi(g,f)=\Phi(fg+gf,\one)$ for all $f,g\in \I$.
Consequently,
\[
\begin{aligned}
\Psi(f,g)+\Psi(g,f)
&=\Phi(f,g)+\Phi(g,f)
-\Phi(fg,\one)-\Phi(gf,\one)\\
&=\Phi(fg+gf,\one)-\Phi(fg+gf,\one)=0.
\end{aligned}
\]
For the last statement, since $e_x(\one-e_x)=0=(\one-e_x)e_x$, so
$0=\Phi(e_x,\one-e_x)=\Phi(e_x,\one)-\Phi(e_x,e_x)$.
Thus
\[
\Psi(e_x,e_x)
=\Phi(e_x,e_x)-\Phi(e_x^2,\one)
=\Phi(e_x,e_x)-\Phi(e_x,\one)=0
\]
for all $x\in X$.
\end{proof}

\begin{lemma}\label{lem:psi-chain}
With notations as above, if $x\leq y\leq z$, then
\begin{equation}\label{eq:psi-chain}
\Psi(e_{xy},e_{yz})=\Psi(e_x,e_{xz}).
\end{equation}
\end{lemma}

\begin{proof}
Since $X$ is a poset, we can assume $x\neq z$ without loss of generality.
Furthermore, we only need to consider $x<y\leq z$.

Suppose $x<y<z$ and set
$p=e_x+e_{xy}$ and $q=e_{yz}-e_{xz}$. A direct computation gives that $pq=qp=0$.
Hence $\Psi(p,q)=0$ by Lemma~\ref{lem:psi-properties}. In other words,
\[
0=\Psi(e_x,e_{yz})-\Psi(e_x,e_{xz})
+\Psi(e_{xy},e_{yz})-\Psi(e_{xy},e_{xz}).
\]
Since $e_x e_{yz}=e_{yz}e_x=0$ and $e_{xy}e_{xz}=e_{xz}e_{xy}=0$, we have
$\Psi(e_x,e_{yz})=0=\Psi(e_{xy},e_{xz})$ by Lemma~\ref{lem:psi-properties} again and hence
obtain the desired \eqref{eq:psi-chain}.

Suppose $x<y=z$ and set $p=e_x+e_{xy}$ and $q=e_y-e_{xy}$. Similarly, we have
\[
0=\Psi(p,q)=\Psi(e_x,e_y)-\Psi(e_x,e_{xy})
+\Psi(e_{xy},e_y)-\Psi(e_{xy},e_{xy}).
\]
We also have that $\Psi(e_x,e_y)$ and $\Psi(e_{xy},e_{xy})$ are zero by Lemma~\ref{lem:psi-properties}.
Therefore $\Psi(e_{xy},e_y)=\Psi(e_x,e_{xy})$, which is
\eqref{eq:psi-chain} for $z=y$.
\end{proof}

We are now ready to establish the main result of this section.

\begin{theorem}\label{thm:incidence-2zpd}
For every finite poset $X$, the incidence algebra $I(X,R)$ is $2$-zpd.
\end{theorem}

\begin{proof}
Let $V$ be an arbitrary $R$-module and
$\Phi:\I\times\I\to V$ be an arbitrary $R$-bilinear map satisfying
$\Phi(f,g)=0$ whenever $fg=gf=0$. Let $\Psi$ be the bilinear map
defined by the equation \eqref{eq:psi-definition}. By Lemma~\ref{lem:commutator-submodule}, we can
define an $R$-linear map $T:[\I,\I]\to V$ by
$T(e_{xy})=\Psi(e_x,e_{xy})$ for every $x<y$.

\smallskip
\noindent\emph{Claim.}
$\Psi(e_{xy},e_{uv})=T([e_{xy},e_{uv}])$, for all $x\leq y$ and  $u\leq v$.

There are three cases appearing.

{\bf Case 1.} $y\neq u$ and $v\neq x$.
At this case $e_{xy}e_{uv}=0=e_{uv}e_{xy}$, then $\Psi(e_{xy},e_{uv})=0$ by Lemma~\ref{lem:psi-properties}.
Moreover, $[e_{xy},e_{uv}]=0$ and hence
$\Psi(e_{xy},e_{uv})=T([e_{xy},e_{uv}])$.

{\bf Case 2.} $y=u$. If $v=x$, then
$x\leq y=u\leq v=x$. The antisymmetry of a partial order implies
$x=y=u=v$. Consequently, by Lemma~\ref{lem:psi-properties},
$\Psi(e_{xy},e_{uv})=\Psi(e_x,e_x)=0=T([e_{xy},e_{uv}])$.

Assume now that $y=u$ and $v\neq x$. Since
$x\leq y\leq v$, the inequality $v\neq x$ implies $x<v$.
On one hand, Lemma~\ref{lem:psi-chain} gives
\[
\Psi(e_{xy},e_{yv})=\Psi(e_x,e_{xv})=T(e_{xv}).
\]
On the other hand, it is clear that $e_{xv}=[e_{xy},e_{yv}]$. Hence
$\Psi(e_{xy},e_{yv})=T([e_{xy},e_{yv}])$.

{\bf Case 3.} $v=x$.
The desired result follows from the {\bf Case 2}, since $\Psi$
is skew-symmetric by Lemma~\ref{lem:psi-properties}.
This proves our claim.

Let $f=\sum_{x\leq y}f(x,y)e_{xy}$ and $g=\sum_{u\leq v}g(u,v)e_{uv}$.
By the just proved Claim, we have
\[
\begin{aligned}
\Psi(f,g)
&=\sum_{x\leq y}\sum_{u\leq v}
  f(x,y)g(u,v)\Psi(e_{xy},e_{uv})\\
&=\sum_{x\leq y}\sum_{u\leq v}
  f(x,y)g(u,v)T([e_{xy},e_{uv}])\\
&=T([f,g]).
\end{aligned}
\]
Recalling the definition of $\Psi$, we have therefore proved
\begin{align}\label{eq11}
\Phi(f,g)=\Phi(fg,\one)+T([f,g]),~\text{for~all}~f,g\in\I.
\end{align}

Comparing the identity \eqref{eq11} and the definition of a $2$-zpd algebra, we need to lift the
linear map $T$ from $[\I,\I]$ to $\I$. Set $\I_{\mathrm{diag}}=R$-$\Span\{e_x \mid x\in X\}$.
It is clear that there is an $R$-module decomposition
\[
\I=\I_{\mathrm{diag}}\oplus[\I,\I].
\]
Let $\pi_{\mathrm{off}}:\I\to[\I,\I]$ be the corresponding
projection and define $\widetilde T=T\circ\pi_{\mathrm{off}}$.
Thus $\widetilde T$ is an $R$-linear extension of $T$ from
$[\I,\I]$ to $\I$ such that $\widetilde T([f,g])=T([f,g])$, for all $f,g\in \I$.

Define two $R$-linear maps $\tau_1,\tau_2:\I\to V$ by
\[
\tau_1(h)=\Phi(h,\one)+\widetilde T(h),
\qquad
\tau_2(h)=-\widetilde T(h).
\]
For any $f,g\in\I$, by \eqref{eq11}, we have
\[
\begin{aligned}
\tau_1(fg)+\tau_2(gf)
&=\Phi(fg,\one)+\widetilde T(fg)-\widetilde T(gf)\\
&=\Phi(fg,\one)+\widetilde T(fg-gf)\\
&=\Phi(fg,\one)+T([f,g])\\
&=\Phi(f,g).
\end{aligned}
\]
Thus $\I$ is $2$-zpd.
\end{proof}

\begin{definition}
An $R$-algebra $\mathcal A$ is called \emph{zero Lie product determined} (zLpd for short), if for every $R$-module $V$ and every
bilinear map $f:\mathcal A\times \mathcal A\to V$ satisfying
\[
\Phi(a,b)=0,~\text{whenever}~[a,b]=0,
\]
then there exists an $R$-linear map $\tau:\mathcal A\to V$ such that
$\Phi(x,y)=\tau([x,y])$ for all $x,y\in \mathcal A$.
\end{definition}

\begin{corollary}
For every finite poset $X$, the incidence algebra $I(X,R)$ is zLpd.
\end{corollary}

The above corollary can be proved similarly as \cite[Proposition 2.1]{BajukBresar22}.
However, it is of independent interest in our opinion, since there are only a few examples of zLpd algebras,
but zLpd algebras are important in Lie theory (see \cite{BresarZPD21}).
We shall study the general structure of zLpd algebras in a future paper.

\section{Functional identity}\label{sec:functional-identities}

Let $F_1, F_2, F_3, F_4: \I \rightarrow \I$ be $R$-linear maps such that
\[
F_1(f)g+fF_2(g)+F_3(g)f+gF_4(f)=0
\]
whenever $fg=gf=0$. This section is devoted to describing the form of $F_i$'s.
As stated in the beginning of Section \ref{sec:two-sided-zpd}, the $R$-bilinear map $\varphi(f,g)=F_1(f)g+fF_2(g)+F_3(g)f+gF_4(f)$
satisfies $\varphi(f,g)=0$, whenever $fg=gf=0$. Symmetrically, there is also $\varphi(g,f)=0$, whenever $fg=gf=0$.
Therefore, in order to achieve our goal, it is natural to study the symmetric version $\frac{1}{2}(\varphi(f,g)+\varphi(g,f))$
and the skew-symmetric version $\frac{1}{2}(\varphi(f,g)-\varphi(g,f))$ separately.
Moreover, it is not surprise to
{\em assume that the commutative ring $R$ contains $\frac{1}{2}$ throughout this section.}

\begin{lemma}\label{lem:antisymmetric-reduction}
Let $A,B:\I\to\I$ be $R$-linear maps such that
\begin{equation}\label{eq:antisymmetric-local}
A(f)g-A(g)f+fB(g)-gB(f)=0,~\text{whenever}~fg=gf=0.
\end{equation}
Then there exists an $R$-linear map $L:\I\to\I$ with
$L(\one)=0$ such that
\begin{equation}\label{eq:antisymmetric-normalized-forms}
A(f)=A(\one)f+L(f),
\qquad
B(f)=fB(\one)+L(f),
\end{equation}
and
\begin{equation}\label{eq:L-commuting-identity}
[L(f),g]+[f,L(g)]=0,~\text{whenever}~[f,g]=0.
\end{equation}
\end{lemma}

\begin{proof}
Let us define $A_0(f):=A(f)-A(\one)f$ and $B_0(f):=B(f)-fB(\one)$, for all $f\in\I$.
Then $A_0(\one)=B_0(\one)=0$ and a direct computation shows $A_0(f)g-A_0(g)f+fB_0(g)-gB_0(f)=0$ whenever $fg=gf=0$.
Define the $R$-bilinear map $\Phi(f,g):=A_0(f)g-A_0(g)f+fB_0(g)-gB_0(f)$.
Then it is clear that $\Phi$ is skew-symmetric and $\Phi(f,g)=0$ whenever $fg=gf=0$.
By Theorem~\ref{thm:incidence-2zpd}, the incidence algebra $\I$ is $2$-zpd, and hence there exist linear maps
$\tau_1,\tau_2:\I \to \I$ such that
\[
\Phi(f,g)=\tau_1(fg)+\tau_2(gf)
\]
for all $f,g\in \I$. Furthermore, the skew-symmetry of $\Phi$ shows
\[
\begin{aligned}
2\Phi(f,g)
&=\Phi(f,g)-\Phi(g,f)\\
&=(\tau_1-\tau_2)(fg-gf).
\end{aligned}
\]
In other words,
\begin{equation}\label{eq:antisymmetric-global-factor}
\Phi(f,g)=\frac{1}{2}(\tau_1-\tau_2)([f,g]), ~\text{for~all}~f,g\in\I.
\end{equation}
Taking $g=\one$ in \eqref{eq:antisymmetric-global-factor}, we have
$$A_0(f)-B_0(f)=\Phi(f,\one)=0$$
for all $f\in \I$, which means $A_0=B_0$. Set $L=A_0=B_0$. Then it is clear that $L(\one)=0$ and
\eqref{eq:antisymmetric-normalized-forms} hold.
Moreover, for all $f,g\in \I$,
\[
\begin{aligned}
\Phi(f,g)
&=L(f)g-L(g)f+fL(g)-gL(f)\\
&=[L(f),g]+[f,L(g)].
\end{aligned}
\]
Therefore, the desired conclusion \eqref{eq:L-commuting-identity}
follows from the equation \eqref{eq:antisymmetric-global-factor}.
\end{proof}

We now determine the form of linear map $L$ appearing in Lemma \ref{lem:antisymmetric-reduction}.

\begin{lemma}\label{lem:commuting-L-structure}
With notations as above, if $|X|>2$ and any two edges in $E(\Gamma(X))$ are contained in one cycle,
then there exist
$s\in R$, a derivation $D:\I\to\I$, and a center-valued map
$\Omega:\I\to Z(\I)$ such that
\begin{equation}\label{eq:L-structure}
L(f)=sf+D(f)+\Omega(f),~\text{for~all}~f\in\I.
\end{equation}
\end{lemma}

\begin{proof}
We determine the action of $L$ on the standard basis element $e_{xy}$ with $x\leq y$.
Let $x\neq y$. One has $[e_x,e_y]=0$ and the identity \eqref{eq:L-commuting-identity} gives
\begin{equation}\label{eq:L-ex-ey}
[L(e_x),e_y]+[e_x,L(e_y)]=0.
\end{equation}
Write $L(e_x)=\sum_{u\leq v}\ell_x(u,v)e_{uv}$.
By \eqref{eq:comm-ex},
\[
[L(e_x),e_y]
=\sum_{u<y}\ell_x(u,y)e_{uy}
-\sum_{y<v}\ell_x(y,v)e_{yv},
\]
whereas
\[
[e_x,L(e_y)]
=-\,[L(e_y),e_x]
=\sum_{x<v}\ell_y(x,v)e_{xv}
-\sum_{u<x}\ell_y(u,x)e_{ux}.
\]
Thus \eqref{eq:L-ex-ey} can be rewritten as
\begin{equation}\label{eq:L-ex-ey-expanded}
\begin{aligned}
0={}&
\sum_{u<y}\ell_x(u,y)e_{uy}
-\sum_{y<v}\ell_x(y,v)e_{yv}+\sum_{x<v}\ell_y(x,v)e_{xv}
-\sum_{u<x}\ell_y(u,x)e_{ux}.
\end{aligned}
\end{equation}
Considering the coefficient of $e_{uy}$ in \eqref{eq:L-ex-ey-expanded} leads to
\begin{equation}\label{eq:L-ex-away}
\ell_x(u,y)=0, ~\text{if}~x\neq u<y\neq x.
\end{equation}
Then we have
\begin{align}
L(e_x)&=\sum_{u< v}\ell_x(u,v)e_{uv}+\sum_{z\in X}\ell_x(z,z)e_{z}\nonumber\\
&=\sum_{u<x}\ell_x(u,x)e_{ux}+\sum_{x<v}\ell_x(x,v)e_{xv}+\sum_{z\in X}\ell_x(z,z)e_{z}.\label{eq:L-exx}
\end{align}

If we further assume $x<y$ in \eqref{eq:L-ex-ey}, then $\ell_x(x,y)+\ell_y(x,y)=0$ by
considering the coefficient of $e_{xy}$ in \eqref{eq:L-ex-ey-expanded}.
Define an element $h_L\in\mathcal I$ by $h_L(x,x)=0$, for all $x\in X$,
and $h_L(x,y)=\ell_y(x,y)=-\ell_x(x,y)$, for all $x<y\in X$.
By \eqref{eq:comm-ex}, we have
\[
\ad_{h_L}(e_x)=[h_L,e_x]=\sum_{u<x}\ell_x(u,x)e_{ux}+\sum_{x<v}\ell_x(x,v)e_{xv}.
\]
Let us define $L_1:=L-\ad_{h_L}$. It follows from equation \eqref{eq:L-exx} that
\begin{equation}\label{eq:L1-ex-diagonal}
L_1(e_x)=\sum_{z\in X}\ell_x(z,z)e_z
\end{equation}
for all $x\in X$. Moreover, it is clear that $L_1(\one)=0$ and
\begin{equation}\label{eq:L1-commuting}
[L_1(f),g]+[f,L_1(g)]=0,~\text{whenever}~[f,g]=0.
\end{equation}

For any $x<y$, if $z\notin\{x,y\}$, then $[e_{xy},e_z]=0$ and the identity \eqref{eq:L1-commuting} gives
\[
[L_1(e_{xy}),e_z]+[e_{xy},L_1(e_z)]=0.
\]
Write $L_1(e_{xy})=\sum_{u\leq v}L_1(e_{xy})(u,v)e_{uv}$. By \eqref{eq:comm-ex} and \eqref{eq:L1-ex-diagonal},
the above equality can be rewritten as
\begin{equation}\label{eq:L1-exy-ez-expanded}
0=\sum_{u<z}L_1(e_{xy})(u,z)e_{uz}
-\sum_{z<v}L_1(e_{xy})(z,v)e_{zv}+\bigl(\ell_z(y,y)-\ell_z(x,x)\bigr)e_{xy}.
\end{equation}
Therefore,
\begin{equation}\label{eq:gamma-away}
\ell_z(x,x)=\ell_z(y,y)~\text{for~all}~z\notin\{x,y\}.
\end{equation}
For a pair $u<v$ with $(u,v)\neq(x,y)$, then at least one of $u,v$ does not belong to the set $\{x,y\}$.
If $v\notin\{x,y\}$, taking $z=v$ in \eqref{eq:L1-exy-ez-expanded}, we get $L_1(e_{xy})(u,v)=0$.
If $v\in\{x,y\}$, then $u\notin\{x,y\}$. Taking $z=u$ in \eqref{eq:L1-exy-ez-expanded}, we also get $L_1(e_{xy})(u,v)=0$. Hence
\begin{align}
L_1(e_{xy})&=\sum_{u< v}L_1(e_{xy})(u,v)e_{uv}+\sum_{z\in X}L_1(e_{xy})(z,z)e_{z}\nonumber\\
&=\sum_{(u,v)\neq(x,y)}L_1(e_{xy})(u,v)e_{uv}+L_1(e_{xy})(x,y)e_{xy}+\sum_{z\in X}L_1(e_{xy})(z,z)e_{z}\nonumber\\
&=L_1(e_{xy})(x,y)e_{xy}+\sum_{z\in X}L_1(e_{xy})(z,z)e_{z}.\label{eq:L1-exy-preliminary}
\end{align}

For a given vertex $z\in X$, notice that $\Gamma(X)\setminus\{z\}$ is connected by Lemma~\ref{lem:graph-equivalences}.
Therefore, for any $u,v\in X\setminus\{z\}$, there is a path in $\Gamma(X)\setminus\{z\}$ from $u$ to $v$, i.e.,
there is a sequence of distinct vertices
\[
u=x_0,x_1,\ldots,x_m=v
\]
such that $\{x_{i-1},x_i\}$ is an edge of $\Gamma(X)\setminus\{z\}$ for every $1\leq i\leq m$.
Now equation \eqref{eq:gamma-away} shows that $\ell_z(u,u)=\ell_z(x_1,x_1)=\cdots=\ell_z(v,v)$.
Hence there exists $\eta_z\in R$ such that
\begin{equation}\label{eq:gamma-off-z}
\ell_z(u,u)=\eta_z,~\text{for~all}~u\neq z.
\end{equation}
It remains to compare the exceptional coefficients $\ell_z(z,z)$. For any $x<y$, set
$p=e_x+e_{xy}$ and $q=e_y-e_{xy}$. On one hand, from \eqref{eq:L1-ex-diagonal} and \eqref{eq:L1-exy-preliminary},
\[
L_1(p)=\sum_{z\in X}\ell_x(z,z)e_z+L_1(e_{xy})(x,y)e_{xy}+\sum_{z\in X}L_1(e_{xy})(z,z)e_z
\]
and
\[
L_1(q)=\sum_{z\in X}\ell_y(z,z)e_z-L_1(e_{xy})(x,y)e_{xy}-\sum_{z\in X}L_1(e_{xy})(z,z)e_z,
\]
which in turn deduces that
\[
[L_1(p),q](x,y)
=\bigl(\eta_x-\ell_x(x,x)\bigr)
+L_1(e_{xy})(x,y)-L_1(e_{xy})(x,x)+L_1(e_{xy})(y,y)
\]
and
\[
[p,L_1(q)](x,y)
=\bigl(\ell_y(y,y)-\eta_y\bigr)-L_1(e_{xy})(x,y)
+L_1(e_{xy})(x,x)-L_1(e_{xy})(y,y).
\]
On the other hand, since $pq=0=qp$, it follows from \eqref{eq:L1-commuting} that
$[L_1(p),q]+[p,L_1(q)]=0$. Therefore,
\begin{equation}\label{eq:gamma-endpoint-relation}
\ell_x(x,x)-\eta_x
=\ell_y(y,y)-\eta_y.
\end{equation}
Set $s_x=\ell_x(x,x)-\eta_x\in R$. Then \eqref{eq:gamma-endpoint-relation} means $s_x=s_y$ for all $x<y$ in $X$.
Since $\Gamma(X)$ is connected, these equalities propagate along every path. There is therefore an $s\in R$ such that
$s_x=s$, for all $x\in X$. Taking \eqref{eq:gamma-off-z} into account, we can rewrite \eqref{eq:L1-ex-diagonal} as
\begin{equation}\label{eq:L1-ex-final}
L_1(e_x)=se_x+\eta_x\one,~\text{for~all}~x\in X.
\end{equation}

For any $x<y$, we next study the coefficients $L_1(e_{xy})(z,z)$ appearing in \eqref{eq:L1-exy-preliminary}, and
claim that $L_1(e_{xy})(u,u)=L_1(e_{xy})(v,v)$ for all $u<v$. There are four cases:

{\bf Case 1.} $u\neq y$, $v\neq x$ and $(u,v)\neq(x,y)$.
Obviously $e_{xy}e_{uv}=e_{uv}e_{xy}=0$ and hence $[L_1(e_{xy}),e_{uv}]+[e_{xy},L_1(e_{uv})]=0$
by \eqref{eq:L1-commuting}. Left multiplication by $e_u$ and right multiplication by $e_v$ in the above equation leads to
\begin{equation}\label{eq:anti-beta-ordinary}
L_1(e_{xy})(u,u)=L_1(e_{xy})(v,v), ~\text{for~all}~ u\neq y,~ v\neq x,~ (u,v)\neq(x,y).
\end{equation}

{\bf Case 2.} $u=y$. This means $x<y<v$. Set $p=e_x+e_{xy}$ and $q=e_{yv}-e_{xv}$. Then $pq=qp=0$ and hence
$[L_1(p),q]+[p,L_1(q)]=0$ by \eqref{eq:L1-commuting}. Left multiplication by $e_y$ and right multiplication by $e_v$ in the above equation leads to
\begin{equation}\label{eq:anti-beta-outgoing}
L_1(e_{xy})(y,y)=L_1(e_{xy})(v,v), ~\text{for~all}~x<y<v.
\end{equation}

{\bf Case 3.} $v=x$. This means $u<x<y$. Set $p=e_y+e_{xy}$ and $q=e_{ux}-e_{uy}$. Then $pq=qp=0$ and hence
$[L_1(p),q]+[p,L_1(q)]=0$ by \eqref{eq:L1-commuting}.
Left multiplication by $e_u$ and right multiplication by $e_x$ in the above equation leads to
\begin{equation}\label{eq:anti-beta-incoming}
L_1(e_{xy})(u,u)=L_1(e_{xy})(x,x), ~\text{for~all}~u<x<y.
\end{equation}

{\bf Case 4.} $(u,v)=(x,y)$. Notice that, by equations \eqref{eq:anti-beta-ordinary}, \eqref{eq:anti-beta-outgoing} and
\eqref{eq:anti-beta-incoming}, we have in fact proved
$L_1(e_{xy})(u,u)=L_1(e_{xy})(v,v)$,
for all $u<v$ with $(u,v)\neq (x,y)$.
By the assumption, there is a cycle containing the edge $\{x,y\}$. In other words,
there exist distinct vertices $x_1,x_2,\cdots,x_n\in X\setminus \{x,y\}$ such that
$x\sim x_1$, $x_1\sim x_2$, $\cdots$, $x_n\sim y$. Therefore
\[
L_1(e_{xy})(x,x)=L_1(e_{xy})(x_1,x_1)=\cdots=L_1(e_{xy})(x_n,x_n)=L_1(e_{xy})(y,y).
\]
We complete the proof of the claim.

Since $\Gamma(X)$ is connected, we can say that the value $L_1(e_{xy})(z,z)$ is independent of $z$, i.e.,
$L_1(e_{xy})(u,u)=L_1(e_{xy})(v,v)$, for any $u,v\in X$. Therefore, there is a scalar $\eta_{xy}\in R$ such that
$\eta_{xy}=L_1(e_{xy})(z,z)$, for all $z\in X$. Then we can rewrite \eqref{eq:L1-exy-preliminary} as
\begin{equation}\label{eq:L1-exy-final}
L_1(e_{xy})=L_1(e_{xy})(x,y)e_{xy}+\eta_{xy}\one, ~\text{for~all}~x<y\in X.
\end{equation}

Up to now, we have determined the action of $L$ on the standard basis elements.
Next we construct a transitive derivation of $\I$. For any $x<y<z$, we set $p=e_x+e_{xy}$ and $q=e_{yz}-e_{xz}$.
Then $pq=qp=0$ and hence $[L_1(p),q]+[p,L_1(q)]=0$ by \eqref{eq:L1-commuting}.
Using \eqref{eq:L1-ex-final} and \eqref{eq:L1-exy-final}, by comparing the coefficients of
$e_{xz}$ in the above identity, we can obtain
\begin{equation}\label{eq:tran}
L_1(e_{xz})(x,z)=L_1(e_{xy})(x,y)+L_1(e_{yz})(y,z)-s.
\end{equation}
Define an element $\sigma\in \mathcal I$ by $\sigma(x,x)=0$, and
\begin{equation}\label{eq:tran2}
\sigma(x,y)=L_1(e_{xy})(x,y)-s
\end{equation}
for all $x<y\in X$. Then it follows from \eqref{eq:tran} and \eqref{eq:tran2} that
$\sigma(x,z)=\sigma(x,y)+\sigma(y,z)$, for all $x\leq y\leq z$.
By Lemma~\ref{lem:transitive-derivation}, the
$R$-linear map $\Delta_\sigma:\I\to\I$
defined by $\Delta_\sigma(e_{xy})=\sigma(x,y)e_{xy}$ is the transitive derivation induced by $\sigma$.

Let us define the $R$-linear map $\Omega:\I\to R\one$ by
$\Omega(e_x)=\eta_x\one$ and $\Omega(e_{xy})=\eta_{xy}\one$ for all $x<y$.
Since $X$ is connected, there is
$R\one=Z(\I)$ by Proposition~\ref{prop:center}. Therefore,
\eqref{eq:L1-ex-final} implies
\[
L_1(e_x)=se_x+\Delta_\sigma(e_x)+\Omega(e_x)
\]
and \eqref{eq:L1-exy-final} and \eqref{eq:tran2} implies
\[
L_1(e_{xy})=se_{xy}+\Delta_\sigma(e_{xy})+\Omega(e_{xy})
\]
for all $x<y$. We can get that
$L_1(f)=sf+\Delta_\sigma(f)+\Omega(f)$ and hence $L(f)=sf+D(f)+\Omega(f)$ for all $f\in \I$,
where $D=\ad_{h_L}+\Delta_\sigma$ is a derivation.
\end{proof}

We would like to remark that if $|X|>2$ and any two edges in $E(\Gamma(X))$ are contained in one cycle,
then Lemma \ref{lem:commuting-L-structure} has proved that an $R$-linear map of $\I$ which is Lie derivable at zero is a generalized Lie derivation.
Now we solve the skew-symmetric local FI \eqref{eq:antisymmetric-local}.

\begin{proposition}\label{thm:antisymmetric-standard-form}
Assume that $|X|>2$ and any two edges in $E(\Gamma(X))$ are contained in one cycle.
Then the $R$-linear maps $A,B:\I\to\I$ satisfy \eqref{eq:antisymmetric-local} if and only if there exist $C_1,C_2\in\I$, a derivation
$D:\I\to\I$, and a linear central-valued map
$\Omega:\I\to Z(\I)$ such that
\begin{equation}\label{eq:antisymmetric-standard-form}
A(f)=C_1f+D(f)+\Omega(f)~\text{and}~
B(f)=fC_2+D(f)+\Omega(f).
\end{equation}
\end{proposition}

\begin{proof}
Assume that the $R$-linear maps $A,B:\I\to\I$ satisfy the local FI \eqref{eq:antisymmetric-local}.
It follows from Lemma~\ref{lem:antisymmetric-reduction} and Lemma~\ref{lem:commuting-L-structure} that
$A,B$ are of the desired form \eqref{eq:antisymmetric-standard-form},
where $C_1=A(\one)+s\one$ and $C_2=B(\one)+s\one$.
Conversely, if $A,B$ are of the form \eqref{eq:antisymmetric-standard-form},
a direct computation shows that they are solutions of the local FI \eqref{eq:antisymmetric-local}.
\end{proof}

Having treated the skew-symmetric local FI \eqref{eq:antisymmetric-local}, we now turn to its
symmetric counterpart.

\begin{lemma}\label{lem:symmetric-reduction}
Let $G,H: \I\to\I$ be $R$-linear maps such that
\begin{equation}\label{eq:symmetric-local}
G(f)g+fH(g)+G(g)f+gH(f)=0,~\text{whenever}~fg=gf=0.
\end{equation}
Define
$G_0(f):=G(f)-G(\one)f$ and $H_0(f):=H(f)-fH(\one)$,
and put
\begin{equation}\label{eq:delta-rho-definition}
\delta=G_0+H_0,
\quad
\rho=G_0-H_0.
\end{equation}
Then $\delta(\one)=\rho(\one)=0$ and for all $f,g\in\I$
\begin{equation}\label{eq:symmetric-master}
\begin{aligned}
\delta(f\circ g)=\delta(f)\circ g+f\circ\delta(g)+[\rho(f),g]+[\rho(g),f].
\end{aligned}
\end{equation}
\end{lemma}

\begin{proof}
The definitions give $G_0(\one)=H_0(\one)=0$, and hence $\delta(\one)=\rho(\one)=0$.
If $fg=gf=0$, then $G(\one)fg+fgH(\one)+G(\one)gf+gfH(\one)=0$, and we have
$$G_0(f)g+fH_0(g)+G_0(g)f+gH_0(f)=0.$$
Define the $R$-bilinear map $\Phi(f,g):=G_0(f)g+fH_0(g)+G_0(g)f+gH_0(f)$.
Then it is clear that $\Phi$ is symmetric and $\Phi(f,g)=0$ whenever $fg=gf=0$.
By Theorem~\ref{thm:incidence-2zpd}, the incidence algebra $\I$ is $2$-zpd, and hence there exist linear maps
$\tau_1,\tau_2:\I \to \I$ such that
\begin{align}
2\Phi(f,g)&=\Phi(f,g)+\Phi(g,f)\nonumber\\
&=(\tau_1+\tau_2)(fg+gf)\nonumber\\
&=\Phi(f\circ g,\one)\label{eq:symmetric-polarization}.
\end{align}
Since $G_0(\one)=H_0(\one)=0$, it follows from the definition of $\Phi$ and \eqref{eq:delta-rho-definition} that
\begin{equation}\label{eq:fg}
\Phi(f\circ g,\one)=G_0(f\circ g)+H_0(f\circ g)=\delta(f\circ g).
\end{equation}
Notice that $\frac{1}{2}\in R$ and $G_0=\frac{1}{2}(\delta+\rho)$,
$H_0=\frac{1}{2}(\delta-\rho)$. By the definition of $\Phi$, we get
\[
\begin{aligned}
2\Phi(f,g)
={}&\delta(f)g+\rho(f)g
+f\delta(g)-f\rho(g)\\
&+\delta(g)f+\rho(g)f
+g\delta(f)-g\rho(f)\\
={}&\delta(f)\circ g+f\circ\delta(g)
+[\rho(f),g]+[\rho(g),f].
\end{aligned}
\]
Combining this identity with \eqref{eq:symmetric-polarization} and \eqref{eq:fg}, we can prove the desired
\eqref{eq:symmetric-master}.
\end{proof}

We now determine the form of linear maps $\delta,\rho$ appearing in Lemma \ref{lem:symmetric-reduction}.

\begin{lemma}\label{lem:symmetric-master-structure}
With notations as above, if $|X|>2$ and any two edges in $E(\Gamma(X))$ are contained in one cycle,
then $\delta$ is a derivation and
there exist $t\in R$ and a center-valued map
$\Omega_0:\I\to Z(\I)$ such that
\begin{equation}\label{eq:rho-structure}
\rho(f)=tf+\Omega_0(f),~\text{for~all}~f\in\I.
\end{equation}
\end{lemma}

\begin{proof}
We first consider the action of $\delta$ and $\rho$ on $e_x$. As usual, we write $\delta(e_x)=\sum_{u\leq v}\delta(e_x)(u,v)e_{uv}$ and
$\rho(e_x)=\sum_{u\leq v} \rho(e_x)(u,v) e_{uv}$.
Taking $f=g=e_x$ in \eqref{eq:symmetric-master}, we have $2\delta(e_x)=2\delta(e_x)e_x+2e_x\delta(e_x)
+2[\rho(e_x),e_x]$ and hence
\begin{equation*}
\delta(e_x)
=\delta(e_x)e_x+e_x\delta(e_x)+[\rho(e_x),e_x],
\end{equation*}
since $\frac{1}{2}\in R$. By Lemma~\ref{lem:coefficient-formulas}, the above identity can be written as
\begin{equation}\label{eq:delta-ex-master-expanded}
\begin{aligned}
\sum_{u\leq v}\delta(e_x)(u,v)e_{uv}
={}&\sum_{u\leq x}\delta(e_x)(u,x)e_{ux}
+\sum_{x\leq v}\delta(e_x)(x,v)e_{xv}\\
&+\sum_{u<x}\rho(e_x)(u,x)e_{ux}
-\sum_{x<v}\rho(e_x)(x,v)e_{xv}.
\end{aligned}
\end{equation}
It follows from \eqref{eq:delta-ex-master-expanded} that
\begin{equation}\label{eq:delta-ex-away}
\delta(e_x)(u,v)=0, ~\text{if}~x\neq u\leq v\neq x.
\end{equation}
When we study the coefficients of $e_{x}$ in \eqref{eq:delta-ex-master-expanded}, there is
$\delta(e_x)(x,x)=\delta(e_x)(x,x)+\delta(e_x)(x,x)$, i.e., $\delta(e_x)(x,x)=0$. Combining this fact
with \eqref{eq:delta-ex-away}, we have
\begin{equation}\label{eq:delta-ex-diagonal-zero}
\delta(e_x)(z,z)=0,~\text{for~all}~z\in X.
\end{equation}
Thus \eqref{eq:delta-ex-away} and \eqref{eq:delta-ex-diagonal-zero} implies that
\begin{align}
\delta(e_x)=&\sum_{u\leq v}\delta(e_x)(u,v)e_{uv}\nonumber\\
=&\sum_{u<x}\delta(e_x)(u,x)e_{ux}+\sum_{x<v}\delta(e_x)(x,v)e_{xv}\nonumber\\
&+\sum_{u<v,~u,v\neq x}\delta(e_x)(u,v)e_{uv}+\sum_{z\in X}\delta(e_x)(z,z)e_{z}\nonumber\\
=&\sum_{u<x}\delta(e_x)(u,x)e_{ux}
+\sum_{x<v}\delta(e_x)(x,v)e_{xv}.\label{eq:delta-ex0}
\end{align}
For any $u<x$, considering the coefficients of $e_{ux}$ in \eqref{eq:delta-ex-master-expanded}, we have
$\delta(e_x)(u,x)=\delta(e_x)(u,x)+\rho(e_x)(u,x)$,
and hence
\begin{equation}\label{eq:rho-ex-incoming-zero}
\rho(e_x)(u,x)=0,~\text{for~all}~u<x.
\end{equation}
Similarly,
\begin{equation}\label{eq:rho-ex-outgoing-zero}
\rho(e_x)(x,v)=0,~\text{for~all}~x<v.
\end{equation}
For any $v\neq x$, since $e_x\circ e_v=0$, it follows from the identity \eqref{eq:symmetric-master} that
$$
0=\delta(e_x)\circ e_v+e_x\circ\delta(e_v)+[\rho(e_x),e_v]+[\rho(e_v),e_x].
$$
Left multiplication by $e_u$ and right multiplication by $e_v$ in the above equation leads to
\begin{equation}\label{eq:rho-ex-away-zero}
\rho(e_x)(u,v)=0,~\text{if}~x\neq u<v\neq x.
\end{equation}
Therefore, we have from \eqref{eq:rho-ex-incoming-zero}, \eqref{eq:rho-ex-outgoing-zero} and \eqref{eq:rho-ex-away-zero} that
\begin{align}
\rho(e_x)=&\sum_{u\leq v}\rho(e_x)(u,v)e_{uv}\nonumber\\
=&\sum_{u<x}\rho(e_x)(u,x)e_{ux}+\sum_{x<v}\rho(e_x)(x,v)e_{xv}\nonumber\\
&+\sum_{u<v,~u,v\neq x}\rho(e_x)(u,v)e_{uv}+\sum_{z\in X}\rho(e_x)(z,z)e_{z}\nonumber\\
=&\sum_{z\in X}\rho(e_x)(z,z)e_{z}.\label{eq:rho-ex-diagonal}
\end{align}

For any $x<y$, taking $f=e_x$ and $y=e_y$ in \eqref{eq:symmetric-master}, we get
$0=\delta(e_x)\circ e_y+e_x\circ\delta(e_y)$, by \eqref{eq:rho-ex-diagonal}.
Left multiplication by $e_x$ and right multiplication by $e_y$ in the above equation leads to
$\delta(e_x)(x,y)+\delta(e_y)(x,y)=0$.
Define an element $h_\delta\in\I$ by $h_\delta(x,x)=0$, for all $x\in X$,
and $h_\delta(x,y)=\delta(e_y)(x,y)=-\delta(e_x)(x,y)$, for all $x<y$. By \eqref{eq:comm-ex}, we have
\[
\ad_{h_\delta}(e_x)=[h_\delta, e_x]=\sum_{u<x}\delta(e_x)(u,x)e_{ux}
+\sum_{x<v}\delta(e_x)(x,v)e_{xv}.
\]
Let us define $\delta_1:=\delta-\ad_{h_\delta}$. It follows from equation \eqref{eq:delta-ex0} that
\begin{equation}\label{eq:delta1-ex-zero}
\delta_1(e_x)=0,~\text{for~all}~x\in X.
\end{equation}
Moreover, due to \eqref{eq:symmetric-master}, it is clear that,
\begin{equation}\label{eq:symmetric-master-delta1}
\delta_1(f\circ g)=\delta_1(f)\circ g+f\circ\delta_1(g)+[\rho(f),g]+[\rho(g),f]
\end{equation}
for all $f,g\in \I$.

We next consider the action of $\delta_1$ and $\rho$ on $e_{xy}$. As usual, write
$\delta_1(e_{xy})=\sum_{u\leq v}\delta_1(e_{xy})(u,v)e_{uv}$ and
$\rho(e_{xy})=\sum_{u\leq v}\rho(e_{xy})(u,v)e_{uv}$ for all $x<y$.
Taking $f=e_x$ and $g=e_{xy}$ in \eqref{eq:symmetric-master-delta1}, we have by \eqref{eq:delta1-ex-zero} that
\begin{align}
\delta_1(e_{xy})=& e_x\circ\delta_1(e_{xy})+[\rho(e_x),e_{xy}]+[\rho(e_{xy}),e_x]\nonumber\\
=& 2\delta_1(e_{xy})(x,x)e_x+\sum_{u<x}\delta_1(e_{xy})(u,x)e_{ux}+\sum_{x<v}\delta_1(e_{xy})(x,v)e_{xv}\nonumber\\
&+\bigl(\rho(e_x)(x,x)-\rho(e_x)(y,y)\bigr)e_{xy}\nonumber\\
&+\sum_{u<x}\rho(e_{xy})(u,x)e_{ux}-\sum_{x<v}\rho(e_{xy})(x,v)e_{xv},\label{eq:x-xy-master}
\end{align}
where the second equality depends on Lemma \ref{lem:coefficient-formulas} and \eqref{eq:rho-ex-diagonal}.
Comparing the coefficients of $e_{uv}$, $e_{x}$, $e_{ux}$, $e_{xv}$ and $e_{xy}$ respectively in both sides of \eqref{eq:x-xy-master}
we obtain
\begin{align}
\delta_1(e_{xy})(u,v)&=0,\quad \text{if}~x\neq u\leq v\neq x,\label{eq:b-support-x}\\
\delta_1(e_{xy})(x,x)&=0,\quad  \label{eq:b-xx-zero}\\
\rho(e_{xy})(u,x)&=0,\quad \text{if}~u<x,\label{eq:ccccc1} \\
\rho(e_{xy})(x,v)&=0,\quad \text{if}~x<v~\text{and}~v\neq y,\label{eq:c-ux-zero}\\
\rho(e_{xy})(x,y)&=\rho(e_x)(x,x)-\rho(e_x)(y,y),\quad  \text{if}~x<y.
\label{eq:c-xy-alpha-x}
\end{align}
Taking $f=e_y$ and $g=e_{xy}$ in \eqref{eq:symmetric-master-delta1}, we have
\[
\delta_1(e_{xy})
=e_y\circ\delta_1(e_{xy})
+[\rho(e_y),e_{xy}]
+[\rho(e_{xy}),e_y],
\]
and similarly,
\begin{align}
\delta_1(e_{xy})(u,v)&=0,\quad \text{if}~y\neq u\leq v\neq y,\label{eq:b-support-y}\\
\delta_1(e_{xy})(y,y)&=0,\quad \label{eq:b-yy-zero}\\
\rho(e_{xy})(u,y)&=0,\quad \text{if}~u<y~\text{and}~u\neq x,\label{eq:c-yv-zero}\\
\rho(e_{xy})(y,v)&=0,\quad \text{if}~y<v,\notag \\
\rho(e_{xy})(x,y)&=\rho(e_y)(y,y)-\rho(e_y)(x,x),\quad\text{if}~x<y.
\label{eq:c-xy-alpha-y}
\end{align}
Therefore,
\begin{align}
\delta_1(e_{xy})=&\sum_{u\in X}\delta_1(e_{xy})(u,u)e_{uu}+\sum_{u<v}\delta_1(e_{xy})(u,v)e_{uv}\nonumber\\
=&\sum_{x< v}\delta_1(e_{xy})(x,v)e_{xv}+\sum_{u< x}\delta_1(e_{xy})(u,x)e_{ux}\nonumber\\
=&\delta_1(e_{xy})(x,y)e_{xy},\label{eq:delta1-exy-form}
\end{align}
where the second equality follows from \eqref{eq:b-support-x} and \eqref{eq:b-xx-zero},
and the third equality follows from \eqref{eq:b-support-y}.
Let $z\in X$ such that $z\neq x$ and $z\neq y$. Taking $f=e_z$ and $g=e_{xy}$ in \eqref{eq:symmetric-master-delta1},
by \eqref{eq:delta1-ex-zero} and \eqref{eq:delta1-exy-form}, we have
\[
[\rho(e_{z}),e_{xy}]+[\rho(e_{xy}),e_z]=0,
\]
which in turn implies, by \eqref{eq:rho-ex-diagonal},
\begin{equation}\label{eq:z-xy-rho}
[\rho(e_{xy}),e_z]
+\bigl(\rho(e_{z})(x,x)-\rho(e_{z})(y,y)\bigr)e_{xy}=0.
\end{equation}
Left multiplication by $e_x$ and right multiplication by $e_y$ in \eqref{eq:z-xy-rho} leads to
\begin{equation}\label{eq:alpha-z-edge}
\rho(e_{z})(x,x)=\rho(e_{z})(y,y),~\text{if}~z\neq x<y\neq z.
\end{equation}
Combining \eqref{eq:z-xy-rho} and \eqref{eq:alpha-z-edge}, we get
$\rho(e_{xy})(u,z)=0$, for all $u<z$ with $z\notin\{x,y\}$. Taking \eqref{eq:ccccc1} into account,
the above identity can be strengthened as
\begin{equation}\label{eq:c-uz-zero}
\rho(e_{xy})(u,v)=0,~\text{if}~u<v\neq y.
\end{equation}
Therefore, it follows from \eqref{eq:c-ux-zero}, \eqref{eq:c-yv-zero} and \eqref{eq:c-uz-zero} that
\begin{align}
\rho(e_{xy})=&\sum_{z\in X}\rho(e_{xy})(z,z)e_{z}+\sum_{u< v}\rho(e_{xy})(u,v)e_{uv}\nonumber\\
=&\sum_{z\in X}\rho(e_{xy})(z,z)e_{z}+\sum_{x\neq u<v\neq y}\rho(e_{xy})(u,v)e_{uv}+\rho(e_{xy})(x,y)e_{xy}\nonumber\\
&+\sum_{x<v\neq y}\rho(e_{xy})(x,v)e_{xv}+\sum_{x\neq u<y}\rho(e_{xy})(u,y)e_{uy}
\nonumber\\
=&\rho(e_{xy})(x,y)e_{xy}+\sum_{z\in X}\rho(e_{xy})(z,z)e_{z}.\label{00}
\end{align}

For a given vertex $z\in X$, notice that $\Gamma(X)\setminus\{z\}$ is connected by Lemma~\ref{lem:graph-equivalences}.
Therefore, for any $u,v\in X\setminus\{z\}$, there is a path in $\Gamma(X)\setminus\{z\}$ from $u$ to $v$.
Then the identity \eqref{eq:alpha-z-edge} can propagate along this path, i.e.,
$\rho(e_{z})(u,u)=\rho(e_{z})(v,v)$ for all $u,v\neq z$. Hence there exists a $\mu_z\in R$ such that
\begin{equation}\label{eq:alpha-off-z}
\rho(e_{z})(u,u)=\mu_z,~\text{for~all}~u\neq z.
\end{equation}
It remains to study the exceptional coefficients $\rho(e_{z})(z,z)$.
Set $t_z=\rho(e_z)(z,z)-\mu_z$, for all $z\in X$. For any $x<y$, combining \eqref{eq:c-xy-alpha-x}, \eqref{eq:c-xy-alpha-y}
and \eqref{eq:alpha-off-z}, we have $t_x=\rho(e_x)(x,x)-\rho(e_x)(y,y)=\rho(e_{xy})(x,y)=t_y$.
Since $\Gamma(X)$ is connected, we can say that the value $t_z$ is independent of $z$, i.e.,
$t_x=t_y$ for all $x,y\in X$. Therefore, there exists a $t\in R$ such that $t_x=t$, for all $x\in X$.
Taking \eqref{eq:alpha-off-z} into account, we can rewrite \eqref{eq:rho-ex-diagonal} as
\begin{equation}\label{eq:rho-ex-final}
\rho(e_x)=te_x+\mu_x\one, ~\text{for~all}~x\in X.
\end{equation}
Furthermore, we have by \eqref{00} that
\begin{equation}\label{eq:rho-exy-with-beta}
\rho(e_{xy})=te_{xy}+\sum_{z\in X}\rho(e_{xy})(z,z)e_z,
\end{equation}
for all $x<y$, where $t=\rho(e_{xy})(x,y)$.

For any $x<y$, we next study the coefficients $\rho(e_{xy})(z,z)$ appearing in \eqref{eq:rho-exy-with-beta},
and claim that $\rho(e_{xy})(u,u)=\rho(e_{xy})(v,v)$ for all $u<v$. There are four cases:

{\bf Case 1.} $u\neq y$, $v\neq x$ and $(u,v)\neq (x,y)$.
Obviously $e_{xy}\circ e_{uv}=0$ and hence $\delta_1(e_{xy})\circ e_{uv}+e_{xy}\circ\delta_1(e_{uv})+[\rho(e_{xy}),e_{uv}]+[\rho(e_{uv}),e_{xy}]=0$
by \eqref{eq:symmetric-master-delta1}.
Using \eqref{eq:delta1-exy-form} and \eqref{eq:rho-exy-with-beta}, a direct computation shows
\begin{equation}\label{eq:beta-ordinary}
\rho(e_{xy})(u,u)=\rho(e_{xy})(v,v), ~\text{for~all}~ u\neq y,~ v\neq x,~ (u,v)\neq(x,y).
\end{equation}

{\bf Case 2.} $u=y$. This means $x<y<v$. Taking $f=e_{xy}$ and $g=e_{yv}$ in \eqref{eq:symmetric-master-delta1}, we have
$$\delta_1(e_{xv})=\delta_1(e_{xy})\circ e_{yv}+e_{xy}\circ\delta_1(e_{yv})+[\rho(e_{xy}),e_{yv}]+[\rho(e_{yv}),e_{xy}].$$
Using \eqref{eq:delta1-exy-form} and \eqref{eq:rho-exy-with-beta} again, we get
\begin{equation*}
\begin{aligned}
\delta_1(e_{xv})(x,v)e_{xv}
={}&(\delta_1(e_{xy})(x,y)+\delta_1(e_{yv})(y,v))e_{xv}
\\
&+\bigl(\rho(e_{xy})(y,y)-\rho(e_{xy})(v,v)\bigr)e_{yv}\\
&+\bigl(\rho(e_{yv})(x,x)-\rho(e_{yv}(y,y))\bigr)e_{xy},
\end{aligned}
\end{equation*}
which in turn shows
\begin{align}
\rho(e_{xy})(y,y)&=\rho(e_{xy})(v,v), ~\text{for~all}~x<y<v,
\label{eq:beta-outgoing}\\
\delta_1(e_{xv})(x,v)&=\delta_1(e_{xy})(x,y)+\delta_1(e_{yv})(y,v), ~\text{for~all}~x<y<v.
\label{eq:lambda-right}
\end{align}

{\bf Case 3.} $v=x$. This means $u<x<y$. Taking $f=e_{ux}$ and $g=e_{xy}$ in \eqref{eq:symmetric-master-delta1},
similar to Case 2, we can get
\begin{align}
\rho(e_{xy})(u,u)&=\rho(e_{xy})(x,x), ~\text{for~all}~u<x<y. \label{eq:beta-incoming}
\end{align}

{\bf Case 4.} $(u,v)=(x,y)$. Notice that, by equations \eqref{eq:beta-ordinary}, \eqref{eq:beta-outgoing} and \eqref{eq:beta-incoming},
we have in fact proved $\rho(e_{xy})(u,u)=\rho(e_{xy})(v,v)$, for all $u<v$ with $(u,v)\neq(x,y)$.
By the assumption, there is a cycle containing the edge $\{x,y\}$.
In other words, there exist distinct vertices $x_1,x_2,\cdots,x_n\in X\setminus \{x,y\}$ such that
$x\sim x_1$, $x_1\sim x_2$, $\cdots$, $x_n\sim y$. Therefore
\[
\rho(e_{xy})(x,x)=\rho(e_{xy})(x_1,x_1)=\cdots=\rho(e_{xy})(x_n,x_n)=\rho(e_{xy})(y,y).
\]
We complete the proof of the claim.

Since $\Gamma(X)$ is connected, we can say that the value $\rho(e_{xy})(z,z)$ is independent of $z$, i.e.,
$\rho(e_{xy})(u,u)=\rho(e_{xy})(v,v)$, for any $u,v\in X$. Therefore, there is a scalar $\mu_{xy}\in R$ such that
$\mu_{xy}=\rho(e_{xy})(z,z)$, for all $z\in X$. Then we can rewrite \eqref{eq:rho-exy-with-beta} as
\begin{equation}\label{eq:rho-exy-final}
\rho(e_{xy})=te_{xy}+\mu_{xy}\one.
\end{equation}

Up to now, we have determined the action of $\delta$ and $\rho$ on the standard basis elements.
Finally, we prove the desired results. Define an element $\sigma \in \I$ by $\sigma(x,x)=0$ for all $x\in X$, and
$\sigma(x,y)=\delta_1(e_{xy})(x,y)$ for all $x<y\in X$.
Then it follows from \eqref{eq:lambda-right} that
$\sigma(x,z)=\sigma(x,y)+\sigma(y,z)$, for all $x\leq y\leq z$.
Taking \eqref{eq:delta1-ex-zero} and \eqref{eq:delta1-exy-form} into account, we have
$\delta_1(e_{xy})=\sigma(x,y)e_{xy}$, for all $x\leq y\in X$. Therefore, by Lemma~\ref{lem:transitive-derivation},
$\delta_1$ is the transitive derivation induced by $\sigma$ and hence $\delta=\ad_{h_\delta}+\delta_1$ is a derivation.

Let us define the $R$-linear map $\Omega_0:\I\to R\one$ by
$\Omega_0(e_x)=\mu_x\one$ and $\Omega_0(e_{xy})=\mu_{xy}\one$ for all $x<y$.
Since $X$ is connected, there is $R\one=Z(\I)$ by Proposition \ref{prop:center}.
It follows from \eqref{eq:rho-ex-final} and \eqref{eq:rho-exy-final} that
$\rho(f)=tf+\Omega_0(f)$ for all $f\in\I$.
\end{proof}

\begin{proposition}\label{thm:symmetric-standard-form}
Assume that $|X|>2$ and any two edges in $E(\Gamma(X))$ are contained in one cycle.
The the $R$-linear maps $G,H:\I\to\I$ satisfy \eqref{eq:symmetric-local} if and only if
there exist $C_1,C_2\in\I$, a derivation
$D:\I\to\I$, and a linear central-valued map
$\Omega:\I\to Z(\I)$ such that
\begin{equation}\label{eq:symmetric-standard-form}
G(f)=C_1f+D(f)+\Omega(f)~\text{and}~
H(f)=fC_2+D(f)-\Omega(f).
\end{equation}
\end{proposition}

\begin{proof}
Assume that the $R$-linear maps $G,H:\I\to\I$ satisfy the local FI \eqref{eq:symmetric-local}.
It follows from Lemma~\ref{lem:symmetric-reduction} and Lemma~\ref{lem:symmetric-master-structure} that
\[
\begin{aligned}
G_0(f)
&=\frac12\delta(f)+\frac t2f+\frac12\Omega_0(f),\\
H_0(f)
&=\frac12\delta(f)-\frac t2f-\frac12\Omega_0(f),
\end{aligned}
\]
for all $f\in \I$. Set $D=\frac{1}{2}\delta$, $\Omega=\frac{1}{2}\Omega_0$,
$C_1=G(\one)+\frac{t}{2}\one$ and $C_2=H(\one)-\frac{t}{2}\one$.
Then $G,H$ are of the desired form \eqref{eq:symmetric-standard-form}.
Conversely, if $G,H$ are of the form \eqref{eq:symmetric-standard-form},
a direct computation shows that they are solutions of the local FI \eqref{eq:symmetric-local}.
\end{proof}

Now we are ready to prove the main theorem.

\begin{theorem}\label{thm:main-characterization}
Let $R$ be a commutative ring with unity such that $\frac{1}{2}\in R$.
Let $X$ be a connected finite poset with $|X|>2$ and set $\I=I(X,R)$.
The following statements are equivalent:
\begin{enumerate}
\item[(i)] Any two edges in $E(\Gamma(X))$ are contained in one cycle.
\item[(ii)] Let $F_1,F_2,F_3,F_4:\I\to\I$ be $R$-linear maps satisfying
\begin{equation}\label{eq:four-map-local}
F_1(f)g+fF_2(g)+F_3(g)f+gF_4(f)=0,
\end{equation}
whenever $fg=gf=0$. Then there exist
$C_1,C_2,C_3,C_4\in\I$, derivations $D_1,D_2:\I\to\I$, and linear
central-valued maps $\Omega_1,\Omega_2:\I\to Z(\I)$ such that
\begin{equation}\label{eq:four-map-standard}
\begin{aligned}
F_1(f)&=C_1f+D_1(f)+\Omega_1(f),\\
F_2(f)&=fC_2+D_1(f)+\Omega_2(f),\\
F_3(f)&=C_3f+D_2(f)-\Omega_2(f),\\
F_4(f)&=fC_4+D_2(f)-\Omega_1(f),
\end{aligned}
\end{equation}
for all $f\in\I$.
\end{enumerate}
\end{theorem}

\begin{proof}
Assume that (i) holds. Let $F_1,F_2,F_3,F_4$ be $R$-linear maps satisfying
\eqref{eq:four-map-local}. Since the condition $fg=gf=0$ is
symmetric in $f$ and $g$, interchanging them gives
\begin{equation}\label{eq:four-map-swapped}
F_1(g)f+gF_2(f)+F_3(f)g+fF_4(g)=0.
\end{equation}
Subtracting \eqref{eq:four-map-swapped} from \eqref{eq:four-map-local} gives
$$0=(F_1-F_3)(f)g-(F_1-F_3)(g)f+f(F_2-F_4)(g)-g(F_2-F_4)(f).$$
Applying Proposition~\ref{thm:antisymmetric-standard-form} to
$A=F_1-F_3$ and $B=F_2-F_4$, there are $P_1,P_2\in\I$, a derivation $\Delta_1$, and a
central-valued map $\Lambda_1$ such that, for all $f\in\I$,
\begin{equation}\label{eq:difference-forms}
\begin{aligned}
(F_1-F_3)(f)&=P_1f+\Delta_1(f)+\Lambda_1(f),\\
(F_2-F_4)(f)&=fP_2+\Delta_1(f)+\Lambda_1(f).
\end{aligned}
\end{equation}
Adding \eqref{eq:four-map-swapped} to \eqref{eq:four-map-local} gives
$$0=(F_1+F_3)(f)g+f(F_2+F_4)(g)+(F_1+F_3)(g)f+g(F_2+F_4)(f).$$
Applying Proposition~\ref{thm:symmetric-standard-form} to
$G=F_1+F_3$ and $H=F_2+F_4$, there are $P_3,P_4\in\I$, a derivation $\Delta_2$, and a
central-valued map $\Lambda_2$ such that
\begin{equation}\label{eq:sum-forms}
\begin{aligned}
(F_1+F_3)(f)&=P_3f+\Delta_2(f)+\Lambda_2(f),\\
(F_2+F_4)(f)&=fP_4+\Delta_2(f)-\Lambda_2(f).
\end{aligned}
\end{equation}
Set
\[
\begin{aligned}
C_1&=\tfrac12(P_1+P_3),&
C_2&=\tfrac12(P_2+P_4),&
C_3&=\tfrac12(P_3-P_1),&
C_4&=\tfrac12(P_4-P_2),\\
D_1&=\tfrac12(\Delta_1+\Delta_2),&
D_2&=\tfrac12(\Delta_2-\Delta_1),\\
\Omega_1&=\tfrac12(\Lambda_1+\Lambda_2),&
\Omega_2&=\tfrac12(\Lambda_1-\Lambda_2).
\end{aligned}
\]
Then it is clear that $D_1,D_2$ are derivations and $\Omega_1,\Omega_2$ are central-valued maps.
Combining \eqref{eq:difference-forms} and \eqref{eq:sum-forms}, we have
\[
\begin{aligned}
F_1(f)
&=\frac12\bigl((F_1-F_3)(f)+(F_1+F_3)(f)\bigr)=C_1f+D_1(f)+\Omega_1(f),\\
F_2(f)
&=\frac12\bigl((F_2-F_4)(f)+(F_2+F_4)(f)\bigr)=fC_2+D_1(f)+\Omega_2(f),\\
F_3(f)
&=\frac12\bigl((F_1+F_3)(f)-(F_1-F_3)(f)\bigr)=C_3f+D_2(f)-\Omega_2(f),\\
F_4(f)
&=\frac12\bigl((F_2+F_4)(f)-(F_2-F_4)(f)\bigr)=fC_4+D_2(f)-\Omega_1(f),
\end{aligned}
\]
for all $f\in\I$. This proves \eqref{eq:four-map-standard} and we get (ii).

Now assume that (ii) holds. We assume opposite that there are two edges in $E(\Gamma(X))$ that are not contained in a cycle.
Then $\Gamma(X)$ has a cut vertex $z$ by Lemma~\ref{lem:graph-equivalences}.
Let $X_1,\ldots,X_m$ be the connected components of $X\setminus\{z\}$, where $m\geq2$.
Set $p_i=\sum_{x\in X_i}e_x$ and let
\begin{equation*}
\mathcal{J}_i
=R\text{-}\Span\{e_{uv}\mid u\leq v,\ u,v\in X_i\cup\{z\},\,
(u,v)\neq(z,z)\}.
\end{equation*}
We next establish some facts about these $R$-submodules $\mathcal{J}_i$ of $\I$.

\smallskip
\noindent\emph{Claim 1.}
If $i\neq j$, then
\begin{equation}\label{eq:Ji-cross-zero}
\mathcal J_i\mathcal J_j=0=\mathcal J_j\mathcal J_i.
\end{equation}

We only prove $\mathcal J_i\mathcal J_j=0$. Suppose, to the contrary, that $\mathcal J_i\mathcal J_j\neq0$.
Then there exist $e_{xy}\in\mathcal J_i$ and $e_{uv}\in\mathcal J_j$ such that
$e_{xy}e_{uv}\neq 0$, which implies $y=u$. Notice that $y\in X_i\cup\{z\}$ and $u\in X_j\cup\{z\}$, but
$X_i$ and $X_j$ are distinct connected components of $X\setminus\{z\}$.
Therefore, $(X_i\cup\{z\})\cap(X_j\cup\{z\})=\{z\}$ and hence $y=u=z$. On the other hand,
by the definition $(x,y)\neq (z,z)$, but $y=z$, we can get that $x\in X_i$.
Similarly, there is $v\in X_j$. That means $x,z,v$ are three distinct vertices and $x<z<v$.
Then $x<v$ and hence the vertices $x$ and $v$ are
adjacent in the comparability graph $\Gamma(X)$. Moreover, the edge $\{x,v\}$ belongs to
$\Gamma(X)\setminus\{z\}$. This edge connects a vertex of $X_i$ to
a vertex of $X_j$, contradicting the fact that $X_i$ and $X_j$ are
distinct connected components of $\Gamma(X)\setminus\{z\}$.
Therefore $\mathcal J_i\mathcal J_j=0$.

\smallskip
\noindent\emph{Claim 2.}
For all $h\in\mathcal J_i$ with $1\leq i\leq m$,
\begin{equation}\label{eq:local-unit}
(p_i+e_z)h=h=h(p_i+e_z).
\end{equation}

In fact, since $X_i\cap \{z\}=\emptyset$, we have
\[
(p_i+e_z)e_{uv}=\left(\sum_{x\in X_i}e_x+e_z\right)e_{uv}=e_{uv},
\]
for any basis element $e_{uv}\in \mathcal J_i$. Therefore, we obtain $(p_i+e_z)h=h$, and
similarly $h=h(p_i+e_z)$, for all $h\in\mathcal J_i$ with $1\leq i\leq m$.

\smallskip
We next construct an $R$-linear map which is Lie derivable at zero.
Notice that every basis element $e_{xy}$ of $\I$, except $e_z$, belongs to exactly one $\mathcal{J}_i$.
Hence there is a direct sum decomposition of $R$-modules as follows
\begin{equation}\label{eq:I-cut-decomposition}
\I=Re_z\oplus\bigoplus_{i=1}^m\mathcal J_i.
\end{equation}
Let $t_1,\ldots,t_m\in R$ be scalars, which are not all equal. We define an $R$-linear map $L:\I\to\I$ by
\[
L(e_z)=-\sum_{i=1}^m t_ip_i
\]
and
\[
L(e_{uv})=t_i e_{uv}
\]
for all $e_{uv}\in \mathcal{J}_i$ and $1\leq i\leq m$. For arbitrary elements $f,g\in \I$,
we write them according to \eqref{eq:I-cut-decomposition} as
$f=\alpha e_z+\sum_{i=1}^m f_i$ and $g=\beta e_z+\sum_{i=1}^m g_i$, where
$f_i,g_i\in\mathcal J_i$ and $\alpha,\beta\in R$.
Then it follows from the definition of $L$ that
\begin{equation}\label{eq:L-cut-formula}
L(f)=\sum_{i=1}^m t_i(f_i-\alpha p_i)~\text{and}~
L(g)=\sum_{i=1}^m t_i(g_i-\beta p_i).
\end{equation}
By Claim~1, a direct computation shows
\begin{equation}\label{eq:commutator-cut-components}
[f,g]=\sum_{i=1}^m h_i,
\end{equation}
where
$h_i=[f_i,g_i]+\alpha[e_z,g_i]+\beta[f_i,e_z]\in\mathcal J_i$.
Using \eqref{eq:L-cut-formula} and Claim~1, we also have
\[
\begin{aligned}
{}[L(f),g]
&=\sum_{i=1}^m t_i
[f_i-\alpha p_i,\,\beta e_z+g_i]\\
&=\sum_{i=1}^m t_i
\bigl([f_i,g_i]+\beta[f_i,e_z]-\alpha[p_i,g_i]\bigr)
\end{aligned}
\]
and
\[
\begin{aligned}
{}[f,L(g)]
&=\sum_{i=1}^m t_i
[\alpha e_z+f_i,\,g_i-\beta p_i]\\
&=\sum_{i=1}^m t_i
\bigl([f_i,g_i]+\alpha[e_z,g_i]-\beta[f_i,p_i]\bigr).
\end{aligned}
\]
Notice that the Claim~2 implies $[p_i+e_z,h]=0$, for any $h\in\mathcal J_i$. Hence
$[p_i,g_i]=-[e_z,g_i]$ and $[f_i,p_i]=-[f_i,e_z]$. Therefore,
\begin{equation}\label{eq:L-cut-commuting-calculation}
[L(f),g]+[f,L(g)]=2\sum_{i=1}^m t_i h_i.
\end{equation}
If we assume $[f,g]=0$, then it follows from \eqref{eq:commutator-cut-components} and \eqref{eq:I-cut-decomposition}
that $h_i=0$ for each $1\leq i\leq m$. Hence $L$ is Lie derivable at zero by \eqref{eq:L-cut-commuting-calculation}, i.e.,
\begin{equation}\label{eq:counterexample-commuting}
[L(f),g]+[f,L(g)]=0,~\text{whenever}~[f,g]=0.
\end{equation}

Finally, we prove the statement (i).
Set $F_1=F_2=L$ and $F_3=F_4=-L$.
If $fg=gf=0$, then $[f,g]=0$, and hence
\[
\begin{aligned}
&F_1(f)g+fF_2(g)+F_3(g)f+gF_4(f)\\
&=L(f)g+fL(g)-L(g)f-gL(f)\\
&=[L(f),g]+[f,L(g)]=0
\end{aligned}
\]
by \eqref{eq:counterexample-commuting}. By our assumption, there exist $C\in\I$,
a derivation $D:\I\to \I$ and a central-valued map $\Omega: \I\to Z(\I)$ such that
\begin{equation}\label{eq:L-assumed-standard}
L(f)=Cf+D(f)+\Omega(f)
\end{equation}
for all $f\in\I$. Writing $\one=e_z+\sum_i p_i$, we have from the construction of $L$ that
$L(\one)=-\sum_i t_ip_i+\sum_i t_ip_i=0$. Taking $f=\one$ in \eqref{eq:L-assumed-standard},
we obtain $C=-\Omega(\one)\in Z(\I)$. Since the poset $X$ is connected, it follows from
Proposition~\ref{prop:center} that $C=c\one$ for some $c\in R$.
For any $1\leq i\leq m$, let $x\in X_i$. Taking $f=e_x$ in \eqref{eq:L-assumed-standard}, we have
\begin{equation}\label{eq:diagonal-contradiction}
t_i e_x=ce_x+D(e_x)+\omega_x\one,
\end{equation}
where $\omega_x\in R$ with $\omega_x\one=\Omega(e_x)$. Let $y\neq x$.
Multiplying $e_y$ in \eqref{eq:diagonal-contradiction} from both sides, we can get from Lemma~\ref{lem:derivation-diagonal}
that $\omega_x=0$. Furthermore, there is $t_i=c$, since $D(e_x)(x,x)=0$ also by Lemma~\ref{lem:derivation-diagonal}.
In other words, all the $t_i$'s are equal to each other, contradicting the
choice of the $t_i$ for $1\leq i\leq m$. Thus (i) holds.
\end{proof}

If $R$-linear maps $F_1,F_2,F_3,F_4:\I\to\I$ have the standard form \eqref{eq:four-map-standard},
a direct computation shows that they are solutions of the local FI \eqref{eq:four-map-local}.

\begin{remark}\label{rem:counterexample}
Notice that the condition $|X|>2$ in Theorem~\ref{thm:main-characterization} is crucial in our proof.
For example, in the proof of Claim 1, we need three distinct vertices $x,z,v$ such that $x<z<v$.
Moreover, when $|X|=2$, the local FI \eqref{eq:four-map-local} has a solution which is not of the form \eqref{eq:four-map-standard},
see \cite[Example 4.3]{ArgacGhahramani25}.
\end{remark}

\begin{corollary}\label{cor:upper-triangular-standard-form}
Let $R$ be a commutative ring with unity such that
$\frac{1}{2}\in R$. Let $n\geq3$ and $\mathcal T_n(R)$ be the
algebra of all $n\times n$ upper triangular matrices over $R$.
For any quadruple of $R$-linear maps
\[
F_1,F_2,F_3,F_4:\mathcal T_n(R)\longrightarrow\mathcal T_n(R),
\]
the identity
\[
F_1(A)B+AF_2(B)+F_3(B)A+BF_4(A)=0
\]
holds whenever $AB=BA=0$ if and only if there exist
$C_1,C_2,C_3,C_4\in\mathcal T_n(R)$, derivations
$D_1,D_2:\mathcal T_n(R)\to\mathcal T_n(R)$, and $R$-linear maps
$\Omega_1,\Omega_2: \mathcal T_n(R)\longrightarrow Z(\mathcal T_n(R))$
such that
\[
\begin{aligned}
F_1(A)&=C_1A+D_1(A)+\Omega_1(A),\\
F_2(A)&=AC_2+D_1(A)+\Omega_2(A),\\
F_3(A)&=C_3A+D_2(A)-\Omega_2(A),\\
F_4(A)&=AC_4+D_2(A)-\Omega_1(A),
\end{aligned}
\]
for all $A\in\mathcal T_n(R)$,
\end{corollary}

\smallskip
\noindent\textbf{Conflict of interest.}
On behalf of all authors, the corresponding author states that
there is no conflict of interest.

\end{document}